\documentclass[12pt]{amsart}
\usepackage{amsmath,amsthm,latexsym,amscd,amsbsy,amssymb,url}
\usepackage[all]{xypic}
 \relax

\chardef\bslash=`\\ 

\makeatletter
\def\verbatim{\interlinepenalty\@M \@verbatim
  \leftskip\@totalleftmargin\advance\leftskip2pc
  \frenchspacing\@vobeyspaces \@xverbatim}
\makeatother
\newtheorem{thm}{Theorem}[section]

\newtheorem{lem}[thm]{Lemma}
\newtheorem{pro}[thm]{Proposition}
\newtheorem{defin}[thm]{Definition}

\begin{document}


\title
{On Multivalued Fixed-Point Free Maps on $\mathbb R^n$}
\author{Raushan ~Z.~Buzyakova}
\address{Department of Mathematics and Statistics,
The University of North Carolina at Greensboro,
Greensboro, NC, 27402, USA}
\email{rzbouzia@uncg.edu}


\begin{abstract}{
To formulate our results let $f$ be a continuous multivalued map from
$\mathbb R^n$ to $2^{\mathbb R^n}$ and $k$ a  natural number such that $|f(x)|\leq k$ for all $x$.
We prove that $f$ is fixed-point free if and only if its continuous extension
$\tilde f:\beta \mathbb R^n\to 2^{\beta \mathbb R^n}$ is fixed-point free. 
If one wishes to stay within metric terms, the result 
can be formulated as follows: $f$ is fixed-point free if and only 
if there exists a continuous fixed-point free extension $\bar f: b\mathbb R^n\to 2^{b\mathbb R^n}$ for some 
metric compactificaton $b\mathbb R^n$  of $\mathbb R^n$.
Using the classical notion of colorablity, we prove
that such an $f$ is always colorable. Moreover, a number of colors sufficient to paint 
the graph can be expressed as a function of $n$ and $k$ only. 
The mentioned results also hold if the domain is replaced by any closed subspace
of $\mathbb R^n$ without any changes in the range.
}
\end{abstract}

\maketitle
\markboth{Raushan Z. Buzyakova}{On Multivalued Fixed-Point Free Maps on $\mathbb R^n$}
{ }

\section{Introduction}\label{S:intro}

A series of topological results about fixed-point free maps are motivated by these two classical set-theoretical
statements (see, in particular, \cite{BE}):
\par\bigskip\noindent
{\it
S1. If  $f:X\to X$ is a fixed-point free map, then there exists a finite cover $\mathcal F$ of $X$ such  that $f(F)$ misses $F$ 
for each $F\in {\mathcal F}$; and
\par\medskip\noindent
S2. Let ${\mathcal P}(X)$ be the set  of all non-empty subsets of $X$ and let $f: X\to {\mathcal P}(X)$ be a map with the property that $x\not \in f(x)$. If there exists a natural number $k$ such that $|f(x)|\leq k$ for
all $x\in X$, then there exists a finite cover $\mathcal F$ of $X$ such  that $F$ misses $\cup\{f(x):x\in F\}$ 
for each $F\in {\mathcal F}$.
} 
\par\bigskip\noindent
One of 
the first results of topological nature related to these statements was obtained by Katetov in \cite{K} by translating the first statement
into this form: {\it For a discrete space $X$, a map $f:X\to X$ is fixed-point free if and only if its 
continuous extension $\tilde f:\beta X\to \beta X$ is
fixed-point free.} This topological version suggests that if one wants to have a similar criterion for a non-discrete space $X$ one has to at least demand
that elements of covers $\mathcal F$ in the statements under discussion be closed subsets of $X$. Thus, it is commonly accepted that when working with a topological space $X$ and a continuous map $f$ from a closed
subspace $A$ of  $X$ to $X$,
any closed subset $F$ of $A$ that misses its image under $f$ is called a {\it color of $f$}. If  there exists a finite cover ("coloring") of $X$ by
colors,  $f$ is said to be {\it colorable}.
Katetov's paper \cite{K} and van Dowen's work \cite{D} have made a significant impact on topologists' interest in the topic.  A number of interesting results of topological nature in the direction of the first statement have been published since the mentioned papers
(see, in particular,  \cite[Section 3.2]{VM2} for references). In this paper we consider one of natural topological versions of the second statement for
Euclidean space $\mathbb R^n$ and its hyperspace $2^{\mathbb R^n}$. 
In \cite{bc} it is proved that a continuous fixed-point free map from  a closed subspace of $\mathbb R^n$ to $\mathbb R^n$
is colorable. In this paper we consider fixed-point free multivalued maps on $\mathbb R^n$ and its closed subspaces. To formulate our main results 
we first introduce the necessary terminology related to multivalued maps.

For a topological space $X$, we use symbol $2^X$ to denote the space of all non-empty closed subsets of $X$ endowed 
with the Vietoris topology and symbol ${\rm exp}_k(X)$ to denote the subspace of $2^X$ that consists of only  those
$A\in 2^X$ for which $|A|\leq k$.
A multivalued map $f: X\subset Z\to 2^Z$ is {\it fixed-point free} if $x\not \in f(x)$ for every $x\in X$.
A closed set $F\subset X$ is a  {\it color}
of a continuous map $f$ from $X\subset Z$ to $2^Z$ if $F$ misses $\cup \{f(x):x\in F\}$.
If $X$ can be covered by finitely many colors of $f$ then $f$ is said to be {\it colorable}.
The main results of the paper are  Theorems  \ref{thm:chromaticnumber} and \ref{thm:colorability}, which state that any  continuous fixed-point free map from a closed subspace of $\mathbb R^n$ to ${\rm exp}_k(\mathbb R^n)$ is colorable and
there exists a formula that computes a number of colors sufficient for painting in terms of $n$ and $k$ only. Using the main result 
we also show that a criterion similar to the Katetov's holds for multivalued maps as well. Namely, we show
that a continuous map $f$ from a closed subspace $X$ of $\mathbb R^n$ to ${\rm exp}_k(\mathbb R^n)$ is fixed-point free
if and only if its continuous extension $\tilde f: \beta X\to {\rm exp}_k(\beta\mathbb R^n)$ is fixed-point free.
To justify the requirement on sizes of $f(x)$ in our main results let us consider one simple example.
Define $f$ from $\omega\setminus \{0\}$ to the space of finite subsets of $\omega$
as follows: $f(n) =\{n+1,n+2,...,2n\}$. The map $f$ is continuous because $\omega$ and the space of finite subsets
of $\omega$ are discrete. Since for every $n\in \omega\setminus \{0\}$, all elements of $\{n+1, n+2,...,2n\}$ must be of different color
we conclude that $f$ is not colorable. This example  justifies our requirement  in the main results
that the set  $\{|f(x)|: x\in X\}$ is bounded by a positive integer.
Before we dive into the technical part of the paper we would like to outline  a short transparent argument
of the main results of the paper for a fixed-point free map $f:\mathbb R\to {{\rm exp}}_2(\mathbb R)$.
For this let $f_1(x) =\min f(x)$ and $f_2(x)=\max f(x)$.
Since $f$ is fixed-point free, the maps $f_1$ and  $f_2$
are fixed-point free real-valued maps.
By \cite[Theorem 2.5]{b}, both functions are colorable. If one lets ${\mathcal F}_1$ and ${\mathcal F}_2$
be colorings of these maps then it is easy to verify that the family 
$\{A\cap B: A\subset {\mathcal F}_1,\ B\subset {\mathcal F}_2\}$ is a coloring of $f$. If one wishes to extend
the argument for the case ${{\rm exp}}_3(\mathbb R)$ using a straightforward inductive approach, then one may find 
oneself dealing with open colors or with a single-valued map with the domain being  a proper subset of the range. 
Existence of open colors
in a single-valued case can be easily deduced from the definition of colorability when one deals with
self-maps. However, if one works with maps from a subspace $X$ of $Y$ into $Y$, a work needs to be done.
Nevertheless, the presented argument can be extended  for ${{\rm exp}}_k(\mathbb R)$ for any $k$ with
some more work and suitable references. Although our argument for any $n$ and $k$ that we present in the paper may seem different from the one just described, a closer look will reveal that
it carries the same idea hidden behind technical details naturally arising when dealing with higher dimensions. 
Throughout the paper we will follow standard 
 notation and terminology as in \cite{Eng}.

\section{Results}\label{S:results}

For simplicity, but without loss of generality,  most of our arguments related to $\mathbb R^n$ will be carried out
for $\mathbb R^5$. This will free the letter "n"  for other purposes.
Since throughout our discussion we will juggle several spaces at  a time we agree that by $\bar S$ we denote the
closure of $S$ in $\mathbb R^k$ (where the value of $k$ is always understood from the context) while $cl_X(P)$ will  denote the closure of $P$ in $X$.
\par\bigskip\noindent
A standard neighborhood in $2^X$ will be denoted as 
$$
\langle U_1,...,U_m\rangle = \{A\in 2^X: A\subset U_1\cup ...\cup U_m\ and\ U_i\cap A\not =\emptyset\ for \ all\ i=1,...,m\},
$$
where $U_1,...,U_m$ are open sets of $X$.

\par\bigskip\noindent
By ${{\rm exp}}_n(X)$ we denote the subspace $\{F\in 2^X: |F|\leq n\}$.

\par\bigskip\noindent
\begin{defin}\label{defin:brightcolor}
Let $f$ be a continuous map from $X\subset Z$ to $2^Z$. A closed set $F\subset X$ is a bright color
of $f$ if $F$ misses $cl_{Z}[\cup \{f(x):x\in F\}]$.
\end{defin}

\par\bigskip\noindent
\begin{pro}\label{pro:openbrightcolor}
Suppose $Z$ is normal, $X$ is closed  in $Z$, $f:X\to 2^Z$ is continuous, and $F$ is a bright color of $f$.
Then there exists an open neighborhood $U$ of $F$ such that $cl_{X}(U)$ is a bright color of $f$.
\end{pro}
\begin{proof}
Since $F$ is a bright color of $f$ and $X$ is closed in $Z$, 
the sets $F$ and $cl_Z(\cup \{f(x):x\in F\})$ are disjoint closed sets in $Z$. Since
$Z$ is normal, there exist 
$V$  an open neighborhood of  $cl_{Z}[\cup \{f(x):x\in F\}]$ in $Z$ and $W$ an open neighborhood of
$F$ in $Z$ such that $cl_Z(W)$ misses $cl_Z(V)$.
Consider the open set $\langle V\rangle$ in  $2^Z$. Since $f$ is continuous and $f(F)\subset \langle V\rangle$
there exists an open neighborhood $U$ of $F$ in $X$ such that $U\subset W$ and $f(cl_Z(U))\subset \langle V\rangle$.
Thus, $cl_Z(\cup\{f(x):x\in cl_Z(U)\})$ is in $cl_Z(V)$. The latter misses $cl_Z(W)$ and therefore
$cl_Z(U)$. Therefore, $U$ is as desired.
\end{proof}

\par\bigskip
Proposition \ref{pro:openbrightcolor} implies that if $X$ is closed in $\mathbb R^k$ and
 $\mathcal F$ is an $m$-sized bright coloring of a continuous map $f:X\to {{\rm exp}}_n(\mathbb R^k)$
 then there exists an $m$-sized  open cover $\mathcal U$ of $X$ the closures of whose elements are
 bright colors of $f$. We will use this observation throughout the paper without formally
 referencing it.

\par\bigskip
In the following discussion that leads to the main result we will restrict ourselves to
closed subspace of $\mathbb R^5$. This is done to avoid accumulation of too many variables.
All arguments are valid if one replace "$5$" with any natural number. 

\par\bigskip\noindent
\begin{defin}\label{defin:statementS5n}
$S(5,n)$ denotes the following statement: there exists the smallest integer $K(5,n)$ such that every
continuous fixed-point free map $f$ from a closed subset $X\subset \mathbb R^5$ to ${{\rm exp}}_n(\mathbb R^5)$
is colorable by at most $K(5, n)$ many bright colors.
\end{defin}

\par\bigskip\noindent
\begin{lem}\label{fgh}
Suppose $f,g,h: X\subset Z\to 2^Z$ are maps; and $g$ and $h$ are colorable by at most $N_g$ and $N_h$
(bright) colors, respectively. If $f(x)\subset g(x)\cup h(x)$ for every $x\in X$ then
$f$ is colorable by at most $N_g\cdot N_h$ (bright) colors. 
\end{lem}
\begin{proof}
Let $\mathcal G$ and $\mathcal H$ be bright colorings of $g$ and $h$ of sizes $N_g$ and $N_h$, respectively. Put 
${\mathcal F}=\{G\cap H: G\cap H\not =\emptyset, H\in {\mathcal H}, G\in {\mathcal G}\}$.
Clearly, $|{\mathcal F}|\leq N_g\cdot N_h$.
Since ${\mathcal G}$ and $\mathcal H$ are closed covers of $X$, so is $\mathcal F$.
Fix $F=H\cap G$ in $\mathcal F$. Since, by hypothesis, $f(x)\subset g(x)\cup h(x)$, we conclude that 
$$
cl_Z[\cup\{f(x):x\in H\cap G\}]\subset cl_Z[\cup\{g(x):x\in H\cap G\}]\cup cl_Z[\cup\{h(x):x\in H\cap G\}].
$$
Since $\mathcal G$ and $\mathcal H$ are bright colorings, 
$cl_Z[\cup\{g(x):x\in H\cap G\}]$ and  $cl_Z[\cup\{h(x):x\in H\cap G\}]$ miss $H\cap G$. Therefore, the left side
of the above set inclusion formula misses $H\cap G$ as well, meaning $H\cap G$ is a bright color for $f$.
\end{proof}

\par\bigskip
In what follows, by $\pi_i$ we denote the projection of $\mathbb R^5$ onto its $i$-th coordinate axis.
The next two statements (Lemmas \ref{lem:allsamenumberofmax} and \ref{lem:largefirstprojection}) hold if we replace $\pi_1$ by $\pi_i$ for any $i\in \{1,...,n\}$.
However, we will carry out our arguments for $\pi_1$ for the already mentioned reason
of avoiding unnecessary variables.
\par\bigskip\noindent
\begin{lem}\label{lem:allsamenumberofmax} 
Suppose $n>M\geq 1$; $A$ is closed in $\mathbb R^5$; 
$f:A\to {{\rm exp}}_n(\mathbb R^5)\setminus {{\rm exp}}_{n-1}(\mathbb R^5)$ is continuous and fixed-point free;
$|\{y\in f(x): \pi_1(y) = \max \pi_1(f(x))\}|=M$ for all $x\in A$; and $S(5,n-1)$  is true.
Then $f$ is colorable by at most $[K(5,n-1)]^2$ bright colors.
\end{lem}
\begin{proof}
Define $g$ and $h$ from $A$ to ${{\rm exp}}_{n-1}(\mathbb R^5)$ as follows:
$$
g(x) = \{z\in f(x): \pi_1(z) =\max \pi_1 (f(x))\} \ {\rm and}\ h(x) = f(x)\setminus g(x).
$$
Since $f(x)$ is finite, $\max \pi_1(f(x))$ exists.  Hence $g(x)$ is defined. Since
$|\{y\in f(x): \pi_1(y) = \max \pi_1(f(x))\}|=M$ and $1\leq M<n$, we conclude that $0<|f(x)\setminus g(x)|<n$ and
$0<|g(x)|<n$. Therefore, $g$ and $h$ are well defined functions from $A$ to ${{\rm exp}}_{n-1}(\mathbb R^5)$.
Observe that $f(x)= g(x)\cup h(x)$ for each $x$. Therefore, by Lemma \ref{fgh},   to reach the conclusion of
our lemma we need to show that $g$ and $h$ are colorable by at most $K(5,n-1)$ bright colors each.
Since we assume that $S(5, n-1)$ holds it suffices to show that $g$ and $h$ are continuous and fixed-point free. 
The latter property follows from the facts that $f$ is fixed-point free and $f(x) = g(x)\cup h(x)$.

To prove continuity of $g$ and $h$, fix $x\in A$ and open neighborhoods ${\mathcal U}_{g(x)}$ and ${\mathcal V}_{h(x)}$
of $g(x)$ and $h(x)$ in ${{\rm exp}}_{n-1}(\mathbb R^5)$. We need to find an open neighborhood
of $x$ in $A$ whose image under $g(x)$ and $h(x)$  are in ${\mathcal U}_{g(x)}$ and ${\mathcal V}_{h(x)}$, respectively.
Without loss of generality we may assume that  the selected neighborhoods are standard, that is, in the form
$$
{\mathcal U}_{g(x)}=\langle U_y: y\in g(x)\rangle\ {\rm and}\  {\mathcal V}_{h(x)}=\langle U_y: y\in f(x)\setminus g(x)\rangle,
$$
where $U_y$ is a fixed bounded open neighborhood of $y$ in $\mathbb R^5$ for each $y\in f(x)$. We may also assume that the following properties hold.
	\begin{description}
		\item[\rm P1] $\min \pi_1(\overline U_y)> \max \pi_1 (\overline U_z)$ whenever $y\in g(x)$ and $z\in f(x)\setminus g(x)$.
		\item[\rm P2] $\overline U_y\cap \overline U_z =\emptyset$ for any distinct $y,z\in f(x)$.
	\end{description}
The property P1 can be achieved since by the definition of $g$, the set $\pi_1(g(x))$ is a singleton and 
its element is strictly greater than  any element of 
$\pi_1(f(x)\setminus g(x))$.
Put ${\mathcal W}_{f(x)} = \langle U_y:y\in f(x)\rangle$. Clearly, ${\mathcal W}_{f(x)}$ is an open neighborhood of $f(x)$.
By continuity of $f$, there exists an open $O$ of $x$ in $A$ such that
$f(O)\subset {\mathcal W}_{f(x)}$. To finish the proof of continuity of $g$ and $h$ it suffices to show
that $g(O)\subset {\mathcal U}_{g(x)}$ and $h(O)\subset {\mathcal V}_{h(x)}$. For this fix an arbitrary $x'\in O$.
By the choice of $O$, we have $f(x')\in {\mathcal W}_{f(x)}$. 
Let $f(x') =\{z_1,...,z_n\}\in {\rm exp}_n(\mathbb R^5)\setminus {\rm exp}_{n-1}(\mathbb R^5)$ and 
$\pi_1(z_i) = \max \pi_1(f(x'))$ for any $i=1,...,M$. By the lemma's condition on $M$, we have 
$\pi_1(z_j)<\max \pi_1(f(x))$ for any $j=M+1,...,n$. By P1, we have 
\begin{description}
	\item[\rm P3] $z_i\in \cup \{U_y:y\in g(x)\}$ for any $i=1,...,M$.
\end{description}
By P2 and P3, we have
\begin{description}
	\item[\rm P4] $z_j\in \cup \{U_y:y\in f(x)\setminus g(x)\}$ for any $j=M+1,...,n$.
\end{description}
By P2 and the fact that $f(x')\in {\mathcal W}_{f(x)}=\langle U_y:y\in f(x)\rangle$, we conclude that
each $U_y$,
participating in the definition of ${\mathcal W}_{f(x)}$, contains exactly one $z_i\in f(x')$.
By P3 and P4, we have
$\{z_1,...,z_M\}\in {\mathcal U}_{g(x)}$ and $\{z_{M+1},...,z_n\}\in {\mathcal V}_{h(x)}$.
Since $g(x')=\{z_1,...,z_M\}$ and $h(x')=\{z_{M+1},...,z_n\}$, we are done. 
\end{proof}

\par\bigskip\noindent
\begin{lem}\label{lem:largefirstprojection}
Suppose $n>1$; $A$ is closed in $\mathbb R^5$; 
$f:A\to {{\rm exp}}_n(\mathbb R^5)\setminus {{\rm exp}}_{n-1}(\mathbb R^5)$ is continuous and fixed-point free;
$|\{y\in f(x): \pi_1(y) = \max \pi_1(f(x))\}|>1$ for all $x\in A$; and $S(5,n-1)$  is true.
Then $f$ is colorable by at most $n\cdot [K(5,n-1)]^2$ bright colors. 
\end{lem}
\begin{proof} For $m=1,...,n-1$, define $O_m\subset A$ as follows: $x\in O_m$ if and only if  
$M_x=|\{y\in f(x):\pi_1(y)=\max \pi_1(f(x))\}|<n-m$.

\par\medskip\noindent
{\it Claim. $O_m$ is open.}
\par\smallskip\noindent
To prove the claim, fix $x\in O_m$ and let ${\mathcal V}_{f(x)}=\langle V_y:y\in f(x)\rangle$
be an open neighborhood of $f(x)$ such that the following hold:
\begin{enumerate}
	\item $V_y$ is a bounded neighborhood of $y$ for every $y\in f(x)$;
	\item $V_y\cap V_z=\emptyset$ for any distinct $y,z\in f(x)$; and
	\item $\min \pi_1(\overline V_y) >\max \pi_1(\overline V_z)$ whenever $\pi_1(y)=\max \pi_1(f(x))$
	and $\pi_1(z)<\max \pi_1(f(x))$.
\end{enumerate}
It suffices to show now that $f^{-1}({\mathcal V}_{f(x)})\subset O_m$. For this pick $x'\in f^{-1}({\mathcal V}_{f(x)})$.
We have $f(x')\in {\mathcal V}_{f(x)}$. By (2) and (3), $M_{x'}\leq M_x$. Hence, $M_{x'}<n-m$. By the definition of $O_m$,
$x'\in O_m$. The claim is proved.

\par\medskip
We will construct our coloring inductively. For short put $K=[K(5,n-1)]^2$.
\par\medskip\noindent
{\it Step 1.} Put $A_1 = A\setminus O_1$.
Thus, an element $x$ of $A$ is in $A_1$ if and only if $M_x\geq n-1$. Since $|\pi_1(f(x))|>1$  for every $x\in A$ we
conclude that $M_x=n-1$ for every $x\in A_1$. Therefore, by Lemma \ref{lem:allsamenumberofmax}, there exists a
$K$-sized open cover ${\mathcal U}_1$ of $A_1$ the closures of whose elements are bright colors for $f$.

\par\medskip\noindent
{\it Assumption.}   Assume for $m-1$ an open family ${\mathcal U}_{m-1}$ of size at most
$(m-1)\cdot K$ is defined  and the following conditions are met:
\begin{description}
	\item[\rm P1] $\overline U$ is a bright color of $f$  for every $U\in {\mathcal U}_{m-1}$; and
	\item[\rm P2] If $M_x\geq n-(m-1)$ then $x\in \bigcup {\mathcal U}_{m-1}$.
\end{description}

\par\medskip\noindent
{\it Step $m<n$.} 
Put $A_m = A\setminus [O_m\cup (\bigcup {\mathcal U}_{m-1})]$.
By Claim, the set $A_m$ is closed. Pick any $x\in A_m$.  Then $x\not\in O_m$, meaning that $M_x\geq n-m$. Also
$x\not\in \bigcup {\mathcal U}_{m-1}$, which, by P2, implies $M_x<n-(m-1)$. Thus, $n-m\leq M_x<n-m+1$. Therefore,
$M_x=n-m$. Therefore, by Lemma \ref{lem:allsamenumberofmax}, there exists a
$K$-sized open cover ${\mathcal V}_m$ of $A_m$ the closures of whose elements are bright colors for $f$.
Put  ${\mathcal U}_m = {\mathcal U}_{m-1}\cup {\mathcal V}_m$. Clearly, the size of this family is at most $m\cdot K$.
Let us verify P1 and P2 for $m$. Property P1 holds since ${\mathcal U}_m$ is the union of two families
satisfying P1. For P2, observe that $M_x\geq n-m$ if and only if
$x\not\in O_m$. But ${\mathcal U}_m$ covers the complement of $O_m$.

\par\smallskip\noindent
The construction is complete. It suffices to show now that $A=\bigcup {\mathcal U}_{n-1}$.
For this pick $x\in A$. By the lemma's hypothesis, $M_x>1=n-(n-1)$. By P2, $x\in \bigcup {\mathcal U}_{n-1}$.
\end{proof}

\par\bigskip
The base step in the proof of our main theorem uses the following statement (\cite[Proposition 2.9]{bc}):
{\it If $f$ is a continuous fixed-point free map from a closed subset $X$ of $\mathbb R^m$ to
$\mathbb R^m$ then $f$ is colorable by at most $m+3$ bright colors.}

\par\bigskip\noindent
\begin{thm}\label{thm:chromaticnumber} 
There exists an integer $K(m,n)$ such that every continuous fixed-point free map from a closed
subspace $X$ of $\mathbb R^m$ into ${{\rm exp}}_n(\mathbb R^m)$ is colorable by at most $K(m,n)$ many
bright colors.
\end{thm}
\begin{proof}
As in previous statements, to avoid accumulation of extra variables, let us deal with $m=5$.
Thus, we need to show that $K(5,n)$ exists for every natural number $n$.
By \cite[Proposition 2.9]{bc}, $K(5,1)$ exists. Assume that $K(5,n-1)$ exists.
To prove the conclusion of the theorem for $n$, fix any fixed-point free continuous map
$f$ from a closed subspace $X$ of $\mathbb R^5$ into ${{\rm exp}}_n(\mathbb R^5)$.
Next define $L\subset X$ as follows: $x\in L$ if and only if $|f(x)|<n$. 
Then $L$ is closed and the range of
$f|_L$ is a subset of ${{\rm exp}}_{n-1}(\mathbb R^5)$ (notice that $L$ can be empty). Therefore, by our inductive assumption,
there exists a $K(5,n-1)$-sized open cover ${\mathcal U}_L$ of $L$ the closures of whose elements are bright colors.
Put $E=X\setminus \bigcup {\mathcal U}_L$. Then $E$ is closed and $|f(x)|=n$ for every $x\in E$.

For $1\leq i\leq n$,
define $S_i$ as follows: $x\in S_i$ if and only if $x\in E$ and $|\pi_i(f(x))|=1$. Notice that $S_i$ can be empty.
Clearly, $S_i$ is closed.
Inductively, we will first cover 
$\cup_{i\leq n} S_i =\{x\in E:  |\pi_i(f(x))|=1\ for\ some\ i\}$ by bright colors and then we will cover the rest of $E$.

\par\bigskip\noindent
{\it Step 1.} Put $E_i =\cap_{j\not =i} S_j$. Notice that $E_i$ can be empty. 
Since $|f(x)|=n$ for every $x\in E$ we conclude that
$|\pi_i(f(x))|=n$ for all $x\in E_i$. Since $n>1$, by 
Lemma \ref{lem:allsamenumberofmax}  (with $\pi_1$ replaced by $\pi_i$ and $M=1$), there exists a finite open cover ${\mathcal U}_1$ of $\cup_{i\leq n}E_i$ the closures of whose
elements are bright colors for $f$. 

\par\bigskip\noindent
{\it Assumption.} Suppose an open finite family ${\mathcal U}_{k-1}$, where $1\leq k-1<n$,
is defined and the following hold:
\begin{description}
	\item[\rm P1] For every at most $(k-1)$-sized subset $I$ of $\{1,...,n\}$ the inclusion $\cap_{j\not \in I}S_j \subset \bigcup {\mathcal U}_{k-1}$ holds; 
	\item[\rm P2] $\overline U$ is a bright color for every $U\in {\mathcal U}_{k-1}$.
\end{description}

\par\bigskip\noindent
{\it Step $k<n$.} For every $k$-sized $I\subset \{1,...,n\}$ put $E_I = [\cap_{j\not \in I}S_j]\setminus [\bigcup {\mathcal U}_{k-1}]$.
Pick $i^*\in I$. Then $|I\setminus \{i^*\}| = k-1$. By P1, the set $\bigcup {\mathcal U}_{k-1}$ contains
$\cap \{S_j:j\in I\setminus \{i^*\}\}$. Since $E_I$ misses $\bigcup {\mathcal U}_{k-1}$, we conclude that
$|\pi_{i^*}(f(x))|>1$ for every $x$ in $E_I$. By Lemma \ref{lem:largefirstprojection}, there exists a finite open cover
${\mathcal U}_I$ of $E_I$ the closures of whose elements are bright colors. Put
${\mathcal U}_k =[\cup \{{\mathcal U}_I: I\subset \{1,...,n\},\ |I|=k\}]\cup {\mathcal U}_{k-1}$
Properties P1 and P2 clearly hold for $k$. The construction is complete.

\par\bigskip\noindent
Let us show that $\{x\in E: |\pi_i(f(x))|=1\ for\ some\ i\}$ is covered by ${\mathcal U}_{n-1}$. For this pick any $z$ in this set.
Put $I_z = \{i\leq n: |\pi_i(f(x))|>1\}$. Clearly, $|I_z| <n$. Since $x\in E$ we conclude that $|f(x)|=n$.
Therefore, $I_z\not = \emptyset$. Therefore $z\in \cap_{j\not\in I_z} S_j$. By P1, $z\in \bigcup {\mathcal U}_{n-1}$.

To finish the proof we need to cover $E\setminus [\bigcup {\mathcal U}_{n-1}]$ by bright colors.
For this observe that $E\setminus [\bigcup {\mathcal U}_{n-1}]$ contains only those $x$ for which $|\pi_i(f(x))|=n>1$
for all $i$. Therefore, by Lemma \ref{lem:allsamenumberofmax} (with $M=1$), the set in question is covered by bright colors.
Since we always used Lemmas \ref{lem:allsamenumberofmax} and \ref{lem:largefirstprojection} to construct our coloring, the size of the coloring depends on 
$n$ and number $5$ only. Thus, $K(5,n)$ exists. 
\end{proof}
\par\bigskip
A non-technical version of the above result is the following theorem, where $n$ and $k$ denote any natural
numbers.
\par\bigskip\noindent
\begin{thm}\label{thm:colorability}
Any continuous fixed-point free map from a closed subspace $X$ of $\mathbb R^k$ into ${{\rm exp}}_n(\mathbb R^k)$ is 
brightly colorable.
\end{thm}

\par\bigskip
Observe that if a continuous map $f:X\to 2^Z$ has the property that $f(x)$ is compact in $Z$ for all $x$
then there exists the continuous extension $\tilde f:\beta X\to 2^{\beta Z}$. For our next discussion
put  ${\mathcal K}(X)=\{A\in 2^X: A\ is\ compact\}$.
In \cite{bc} it is proved that a continuous map $f$ from a closed subspace $X$ of $\mathbb R^k$ to $\mathbb R^k$
is fixed-point free if and only if its continuous extension $\tilde f: \beta X\to \beta \mathbb R^k$ is fixed-point free.
It is natural to ask if the corresponding statement holds for a multivalued map 
$f: X\to {{\rm exp}}_n(\mathbb R^k)$ and its continuous  extension $\tilde f: \beta X\to {{\rm exp}}_n(\beta \mathbb R^k)$. Observe first that
the continuous extension exists since  ${{\rm exp}}_n(\beta \mathbb R^k)$ is compact.  The affirmative answer to this question
follows from the next statement.

\par\bigskip\noindent
\begin{pro}\label{pro:ftobetaf}
 Let $X$ be a closed subspace of a normal space $Z$. If $\mathcal F$ is a bright coloring 
 of a continuous map $f:X\to {\mathcal K}(Z)$, then $\{\beta F:F\in {\mathcal F}\}$ is a bright coloring
 of $\tilde f:\beta X\to {\mathcal K}({\beta Z})$.
\end{pro}
\begin{proof}
Since $X$ is normal and $\mathcal F$ is a finite closed cover of $X$ the family $\{\beta F:F\in {\mathcal F}\}$
is a closed cover of $\beta X$. Therefore, we only need to show that $\beta F$ is a bright color for
$\tilde f$. Since $F$ is a bright color for $f$ the set $F$ misses $cl_Z(\cup \{f(x): x\in F\})$.
Since $F$ and  $cl_Z(\cup \{f(x): x\in F\})$ are disjoint closed subsets of the normal space $Z$ we conclude that
$\beta F$ misses $cl_{\beta Z}(\cup \{f(x): x\in F\})$. Since $\tilde f$ is the continuous extension of $f$
we conclude that $cl_{\beta Z}(\cup \{\tilde f(x): x\in \beta F\})=cl_{\beta Z}(\cup \{f(x): x\in F\})$.
Thus $\beta F$ misses $cl_{\beta Z}(\cup \{\tilde f(x): x\in \beta F\})$, whence $\beta F$ is a bright color of $\tilde f$.
\end{proof}

\par\bigskip\noindent
\begin{thm}\label{thm:fpfcriterion}
Let $f$ be a continuous map from a closed subspace $X$ of $\mathbb R^k$ to
${{\rm exp}}_n(\mathbb R^k)$. Then $f$ is fixed-point free if and only if 
$\tilde f:\beta X\to {{\rm exp}}_n(\beta \mathbb R^k)$ is fixed-point free.
\end{thm}
\begin{proof}
Sufficiency is obvious. To prove necessity, let $\mathcal F$ be a bright coloring of $f$.
By Proposition \ref{pro:ftobetaf}, $\{\beta F:F\in {\mathcal F}\}$ is a coloring for $\tilde f$.
Since $\{\beta F:F\in {\mathcal F}\}$ covers $\beta X$ and $f(\beta F)$ misses $\beta F$ for each
$F\in {\mathcal F}$, we conclude that $\tilde f$ does not fix any point.
\end{proof}

\par\bigskip
Using spectral techniques it is observed in \cite[Corollary 3.5.7]{VM2} that if $f$ is a continuous colorable
self-map on a separable metric space $X$ then one can find a continuous fixed-point free extension 
$\bar{f}:bX\to bX$, where $bX$ is a metric compactification of $X$ of the same dimensionality as that of $X$. Using the same technique
we will next outline a proof  for the following metric version of   Theorem \ref{thm:fpfcriterion} for natural numbers
$n$ and $k$.

\par\bigskip\noindent 
\begin{thm}\label{thm:metricversion}
Let $f:\mathbb R^k\to {{\rm exp}}_n(\mathbb R^k)$ be continuous. Then $f$ is fixed-point free
if and only if there exists a continuous fixed-point free extension 
$\bar f:b\mathbb R^k\to {{\rm exp}}_n(b\mathbb R^k)$, where $b\mathbb R^k$ is some metric compactification
of $\mathbb R^k$ of dimension $k$.
\end{thm}
\begin{proof}
Sufficiency is obvious. Let us outline a proof of necessity. By Theorem \ref{thm:fpfcriterion},
$\tilde f:\beta \mathbb R^k\to {{\rm exp}}_n(\beta\mathbb R^k)\subset 2^{\beta\mathbb R^k}$
is fixed-point free. By \v S\v epin spectral theorem \cite{S}, we can find spectra 
$\{b_\alpha(\mathbb R^k), \pi^\gamma_\alpha, {\mathcal A}\}$
and $\{2^{b_\alpha(\mathbb R^k)}, p^\gamma_\alpha, {\mathcal A}\}$
with inverse limits $\beta \mathbb R^k$ and $2^{\beta \mathbb R^k}$, respectively, and a family of maps 
$\{f_\alpha:\alpha\in {\mathcal A}\}$ such that
the following hold:
\begin{enumerate}
	\item $b_\alpha(\mathbb R^k)$ is a metric compactification of $\mathbb R^k$ of dimension $k$ for all $\alpha$
	\item $\pi^\gamma_\alpha$ and $p^\gamma_\alpha$ are identity maps on $\mathbb R^k$ and $2^{\mathbb R^k}$, respectively.
	\item $f_\alpha\circ \pi_\alpha = p_\alpha\circ f$.
\end{enumerate}
Since $\tilde f$ is fixed-point free and $b_\alpha(\mathbb R^k)$ is compact for every $\alpha$
we may assume that $f_\alpha$ is fixed-point free for every $\alpha$.
By (2) and (3), each $f_\alpha$ coincides with $f$ on $\mathbb R^k$. Therefore, each 
$\{f_\alpha, b_{\alpha}(\mathbb R^k)\}$ serves our purpose.
\end{proof}

\par\bigskip
We would like to finish the paper by commenting on a number of colors sufficient to paint a given graph.
If one follows the argument of Theorem \ref{thm:chromaticnumber} or the argument for the reals outlined in
the introduction one will quickly see that the size of the coloring constructed for the case 
${{\rm exp}}_n(\mathbb R^k)$ is at least $k^n$. Therefore, it is natural to wonder if an estimate
for the required number of colors can be represented as a polynomial of both $n$ and $k$.
\par

\end{document}